\documentclass[12pt]{amsart}
\usepackage{amsmath,amsthm,amsfonts,amssymb,verbatim,eucal,bm}
\usepackage[all]{xy}

\numberwithin{equation}{section}

\newtheorem{thm}[equation]{Theorem} 
\newtheorem{prop}[equation]{Proposition}
\newtheorem{lemma}[equation]{Lemma} 

\newtheorem{example}[equation]{Example}
\newtheorem{remark}[equation]{Remark}

\DeclareMathOperator{\Alt}{Alt}

\DeclareMathOperator{\Ext}{Ext}

\DeclareMathOperator{\Ima}{Im}

\DeclareMathOperator{\Span}{Span}

\newcommand{\kk}{\text{\bf k}}
\newcommand{\NN}{\mathbb N}

\newcommand{\DOT}{\setlength{\unitlength}{1pt}\begin{picture}(2.5,2)
               (1,1)\put(2,3.5){\circle*{3}}\end{picture}}

\newcommand{\Hom}{\mbox{\rm Hom\,}}

\newcommand{\ot}{\otimes}

\newcommand{\cH}{\mathcal{H}}
\newcommand{\chlk}{\cH_{\lambda,\kappa}}
\newcommand{\chabl}{\cH_{\lambda,\alpha,\beta}}
\newcommand{\CC}{\mathbb{C}}

\newcommand{\HH}{{\rm HH}}

\begin{document}

\begin{abstract}
We give Braverman-Gaitsgory style conditions for general 
PBW deformations
of skew group algebras 
formed from finite groups acting on Koszul algebras.
When the characteristic divides the order of the group,
this includes deformations of the group action as well as of the Koszul relations.
\end{abstract}
\title[Deforming group actions on Koszul algebras] 
{Deforming group actions on Koszul algebras}

\date{June 22, 2016}
\author{Anne V.\ Shepler}
\address{Department of Mathematics, University of North Texas,
Denton, Texas 76203, USA}
\email{ashepler@unt.edu}
\author{Sarah Witherspoon}
\address{Department of Mathematics\\Texas A\&M University\\
College Station, Texas 77843, USA}\email{sjw@math.tamu.edu}
\thanks{
The second author was partially supported by
NSF grant \#DMS-1401016.}

\maketitle

\section{Introduction}

Braverman and Gaitsgory~\cite{BG}  gave conditions for an 
algebra
to be a PBW deformation of a Koszul algebra.
Etingof and Ginzburg~\cite{EG} adapted these conditions
to the setting of a Koszul ring over a semisimple group ring $\CC G$
using results of Beilinson, Ginzburg, and Soergel~\cite{BGS}
in order to study symplectic reflection algebras. 
These are certain kinds of
deformations of a skew group algebra $\CC [x_1,\ldots, x_{2n}]\rtimes G$ 
that preserve a symplectic group action.
More generally, Drinfeld~\cite{Drinfeld} considered such deformations 
of a skew group algebra $\CC [x_1,\ldots,x_n]\rtimes G$
for $G$ an arbitrary finite group acting linearly.
We showed~\cite{quad} how to adapt the techniques
of Braverman and Gaitsgory
to an algebra defined over a group ring $kG$ that is not necessarily
semisimple,   aiding exploration of deformations
of a skew group algebra $S\rtimes G$ for $S$ any Koszul algebra
and $G$ any finite group.  
There, we  examined deformations  preserving the
action of $G$ on the Koszul algebra $S$. 
However, other types of
deformations are possible, some
arising only in the modular setting, where the 
characteristic of the field $k$ divides the order of $G$.
Here,
we study deformations of $S\rtimes G$
that deform not only the generating relations of the Koszul algebra 
$S$ but also deform the action of $G$ on $S$.  
This construction recollects 
the graded affine Hecke algebras of Lusztig~\cite{Lusztig89}, in which
a group action is deformed; in the nonmodular setting, these
were shown by Ram and the first author~\cite{RamShepler} to be
isomorphic to Drinfeld's deformations. 

Every deformation of an algebra defines a Hochschild
2-cocycle of that algebra. 
A central question in deformation theory asks which
cocycles may be lifted to deformations.
We use homological techniques in this paper to answer this question
in our context: For $S$ any Koszul algebra with action of a finite group $G$,
we show in Theorem~\ref{thm:main} 
that obstructions to lifting cocycles on $S\rtimes G$
correspond to concrete conditions on parameter functions defining
potential PBW deformations.
Such deformations are filtered algebras with associated
graded algebra precisely $S\rtimes G$. 
Our theorem generalizes 
\cite[Theorem 5.4]{quad} to include deformations of the group action. 
It applies to many algebras of interest  that are  
deformations of algebras of the form
$S\rtimes G$.
For example, one might take
$S$ to be the symmetric algebra (polynomial ring) $S(V)$ 
on a finite
dimensional vector space $V$,
or a skew (quantum) polynomial ring 
$S_{\bf q}(V)$ with multiplication skewed
by a tuple ${\bf q} = (q_{ij})$ of scalars, 
or a skew exterior algebra, 
or even the Jordan plane or a Sklyanin algebra.

Our primary tool is a twisted product resolution
constructed by Guccione, Guccione, and Valqui~\cite{GGV}
and adapted in~\cite{quad}.
We use it here to prove Theorem~\ref{thm:hom},  a more homological version
of our main Theorem~\ref{thm:main} from which we prove Theorem~\ref{thm:main}
as a corollary. 
In the nonmodular setting, a simpler resolution
suffices, one that is induced directly from the Koszul resolution 
of $S$ itself.
The twisted product resolution we use here
partitions homological information according to type;
cochains corresponding to deformations of the group action
and to deformations of the Koszul relations live on two distinct parts
of the resolution. Conditions for PBW deformations include interaction
among the parts.  
We obtain explicit conditions in 
the special case that the Koszul
algebra $S$ is a polynomial ring in Theorem~\ref{RawPBWConditions},
generalizing \cite[Theorem~3.1]{ueber}.
Our result may also be proven directly via the Composition-Diamond Lemma,
used by Khare~\cite[Theorem~27]{Khare} for 
deformations of the action of a cocommutative algebra on a polynomial ring. 
An advantage of our approach is that it yields conditions 
much more generally for all Koszul algebras.
When the characteristic does not divide the group order, we 
strengthen \cite[Theorem~4.1]{ueber}
by showing in Theorem~\ref{thm:nonmod} 
that a deformation of the group action and Koszul relations together
is isomorphic to one in which only the Koszul relations are deformed. 
We give an example to show that Theorem~\ref{thm:nonmod} is
false in the modular setting. 

Let $k$ be any field.  We assume the characteristic of $k$
is not 2 throughout to make some results easier to state. 
All tensor products are over $k$ unless otherwise indicated,
that is, $\otimes=\otimes_k$. We assume that in each graded or 
filtered  $k$-algebra,  elements of $k$ have 
degree~$0$.

\section{PBW Deformations of Koszul algebras twisted by groups}\label{sec:PBW}

In this section, we recall some definitions and state our
main result giving Braverman-Gaitsgory style conditions for
PBW deformations. The proof will be given in Section~\ref{sec:hom2} after
we recall and develop the needed homological algebra. 

\subsection*{PBW deformations}
Let $\kk$ be a ring with unity (for example,
the field $k$ or a group ring $kG$).
Let $\cH$ be  a finitely generated filtered $\kk$-algebra, 
so that we may write $\cH=T_{\kk}(U)/(P)$
for some finite dimensional $\kk$-bimodule $U$ and ideal
$(P)$ generated by a subset
$P$ of the tensor algebra $T_{\kk}(U)$ 
consisting of filtered 
elements.
Thus elements of $P$ may be 
nonhomogeneous with respect to the 
grading on the free algebra $T_{\kk}(U)$ 
with $U$ in degree $1$.
An element of $T_{\kk}(U)$ has {\em filtered degree d} if
it lies in the $d$-th filtered piece
$\oplus_{i\leq d} \, (U)^{\otimes_{\kk} i}$
of $T_{\kk}(U)$ 
but not in the $(d+1)$-st.
We associate to any presentation of a filtered algebra a homogenous
version,
$$\text{HomogeneousVersion}\big(T_{\kk}(U)/(P)\big)
=T_{\kk}(U)/(R),$$ 
where 
$R=\cup _d\, \{\pi_d(p): p\in P \text{ of filtered degree } d\}$
and $\pi_d:T_{\kk}(U)\rightarrow (U)^{\otimes_{\kk} d}$ 
projects onto the homogeneous component of degree $d$.

We say that a filtered algebra $\cH$ with a given
presentation 
is a {\em PBW deformation} of its homogeneous version
if it has the {\em PBW property},
i.e., the associated graded algebra of $\cH$ coincides
with the homogeneous version:
$$
\text{Gr}(\cH)\cong \text{HomogeneousVersion}(\cH)
\quad\text{ as graded algebras.}
$$ 
 Given a fixed presentation in terms
of generators and relations, 
we often merely say that $\cH$ is
a PBW deformation.
This terminology originated from the Poincar\'e-Birkhoff-Witt Theorem,
which states that the associated graded algebra of the 
universal enveloping algebra of a Lie algebra is its homogeneous
version, namely, a polynomial ring.

\begin{remark}{\em
The reader is cautioned that
authors use the adjective {\em PBW}
in slightly different ways.
For example,
in Braverman-Gaitsgory~\cite{BG}
and also in~\cite{quad}, 
the homogeneous version of a filtered
{\em quadratic} algebra is defined 
by
projecting every generating relation
onto its degree $2$ part, instead of its
highest homogeneous part.  
This merely means that filtered
relations of degree $1$ must be considered
separately in PBW theorems there.
}\end{remark}

\subsection*{Group twisted Koszul algebras}
Let $S$ be a finitely generated graded Koszul $k$-algebra. 
Then $S$ is a quadratic
algebra generated by some finite dimensional $k$-vector
space $V$ (in degree $1$) 
with generating quadratic relations 
$R$, some $k$-subspace of $V\ot V$: 
$$S=T_k(V)/(R)\, . $$
Let $G$ be a finite group acting by graded automorphisms on $S$.
This is equivalent to $G$ acting linearly 
on $V$ with the relations $R$ preserved set-wise.
We denote the action of $g$ in $G$ on $v$ in $V$ by $^gv$ in $ V$.
The {\em skew group algebra} (or semidirect product algebra)
$S\rtimes G$ (also written $S\# G$) is the $k$-algebra generated
by the group algebra $kG$ and the vector space $V$
subject to the relations given by $R$
together with the relations
$gv-\, ^gvg$ for $g$ in $G$ and $v$ in $V$.
We identify
$S\rtimes G$ with a filtered algebra
over the ring $\kk=kG$ generated by $U=kG\ot V\ot kG$:
$$
S\rtimes G \cong T_{kG}(kG\ot V\ot kG)/(R\cup R')
$$
as graded algebras, 
where elements of $G$ have degree 0 and
elements of $V$ have degree 1,
and where 
\begin{equation}\label{eqn:R-prime}
R'=\Span_k \{g\otimes v\otimes 1 -  1\otimes \, ^gv\otimes g: 
v\in V,\ g\in G \}\subset kG \ot V\ot kG\, .
\end{equation} 
Here we identify $R\subset V\ot V$ with a subspace of $$k\ot V\ot k\ot V\ot k
\subset kG\ot V\ot kG\ot V\ot kG\cong
(kG\ot V\ot kG)\ot _{kG} (kG\ot V\ot kG). $$

\subsection*{PBW deformations of 
group twisted Koszul algebras}
Now 
suppose $\mathcal{H}$ is a PBW deformation 
of $S\rtimes G$.
Then $\mathcal{H}$ is 
generated by $kG$ and $V$
subject to nonhomogeneous relations
of degrees $2$ and $1$ of the form
$$
\begin{aligned}
P&=\{ r-\alpha(r)-\beta(r): r\in R \} \ \ \ \mbox{ and } \\
P'&=\{ r'-\lambda(r'): r'\in R' \}
\end{aligned}
$$%
for some $k$-linear parameter functions
$$\alpha:R\rightarrow  V\otimes kG, \ 
\beta: R\rightarrow  kG, \ 
\lambda:  R'\rightarrow  kG \, . $$
 That is,  $\mathcal{H}$
can be realized as the quotient 
$$
\mathcal{H}=T_{kG}(kG\ot V\ot kG)/ ( P \cup P').
$$
Note we may assume that
$\alpha$ takes values in $V\ot kG\cong k\ot V\ot kG$,
rather than more generally in $kG\ot V\ot kG$,
without changing the $k$-span of $P\cup P'$, since the relations $P'$
allow us to replace elements in $kG\ot V\ot kG$ with those in $k\ot V\ot kG$.

In our main theorem below, we determine which 
such quotients define PBW deformations of $S\rtimes G$.
We first need some notation for
decomposing
any functions $\alpha, \beta, \lambda$ as above.
We identify $\lambda: R'\rightarrow kG$ with the
function 
(of the same name)
$\lambda: kG\otimes V\rightarrow kG$
mapping
$g\ot v$ to $\lambda(g\ot v\ot 1 - 1\ot \, {}^gv\ot g)$
for all $g$ in $G$ and $v$ in $V$.
We write 
$$\alpha(r)=\sum_{g\in G} \alpha_g(r) g ,  \ \ \  
\beta(r) =\sum_{g\in G} \beta_g(r)g, \ \ \
\lambda(h\ot v) = \sum_{g\in G} \lambda_g(h\ot v) g
$$ for functions
$\alpha_g:R\rightarrow V$, 
$\beta_g:R\rightarrow k$, $\lambda_g : kG\ot V\rightarrow k$ (identifying $V$ with $V\ot k$
in $V\ot kG$). 
 Write
$\lambda(g\ot - ):V\rightarrow kG$ for the  function
induced from $\lambda$ by fixing $g$ in $G$.
Let $m:kG\ot kG\rightarrow kG$ be multiplicaton on $kG$
and let $\sigma: kG\ot V\rightarrow V\ot kG$ be 
the twist map given
by
\[
   \sigma ( g\ot v) = {}^gv\ot g
\quad\text{ for } g\text{ in }G, \ v\text{ in }V.
\]

For the statement of the theorem, we 
set

\begin{equation}\label{eqn:Habc}
\begin{aligned}
\chabl = T_{kG}(kG\ot V \ot kG)  
\ /\ \big(r-\alpha(r)-\beta(r),\ r' - \lambda(r'):r\in R, \ r' \in R' \big)
\end{aligned}
\end{equation} 

\medskip

\noindent
for linear parameter functions
$\alpha: R\rightarrow V\ot kG$,
$\beta:   R\rightarrow  kG$,
$\lambda:  R'\rightarrow  kG$ 
and for $R$ the space of Koszul relations and $R'$ the space of
group action relations~(\ref{eqn:R-prime}). 
The functions $\alpha$ and $\beta$ are extended uniquely to 
right $kG$-module homomorphisms from $R\ot kG$
to $V\ot kG$ and $kG$, respectively. 

\newpage

\begin{thm}\label{thm:main}
Let $G$ be a finite group and let $V$ be a $kG$-module.
Let $S=T_k(V)/(R)$ be a Koszul algebra
with subspace $R\subset V\ot V$ closed under the action of $G$.
Then a filtered algebra $\mathcal{H}$
is a PBW deformation of $S\rtimes G$ if and only if
\[
\mathcal{H}\cong \chabl 
\]
for some linear parameter functions
$\alpha:R\rightarrow  V\otimes kG, \ 
\beta: R\rightarrow  kG, \ 
\lambda: kG\ot V \rightarrow  kG $ 
satisfying
\begin{itemize}
\item[(1)]
$ \ 1\ot\lambda-\lambda (m\ot 1)+(\lambda\ot 1) (1\ot\sigma) = 0 $,
\item[(2)]
$ \ \lambda (\lambda\ot 1)-\lambda(1\ot \alpha)=(1\ot\beta)-(\beta\ot 1)
(1\ot\sigma)(\sigma\ot 1)$,
\item[(3)]
$ \ (1\ot\alpha)-(\alpha\ot 1)(1\ot\sigma)(\sigma \ot 1)
=\lambda\ot 1+(1\ot \lambda)(\sigma\ot 1)$,
\item[(4)]
$ \ \alpha  ((1\ot\sigma)(\alpha\ot 1)- 1\ot\alpha) + \sum_{g\in G} \alpha_g\ot \lambda(g\ot -)
=1\ot \beta-\beta\ot 1$,
\item[(5)]\rule{0ex}{2
ex}
$  \beta ((1\ot \sigma)(\alpha\ot 1) - 1\ot\alpha)= - \lambda (\beta\ot 1) $,
\item[(6)]
  $ \ \alpha\ot 1-1\ot\alpha=0$,
\end{itemize}
upon projection of images of the maps to $S\rtimes G$.
Here,
the map in $(1)$ is defined on $kG\ot kG\ot V$, 
the maps in $(2)$ and $(3)$ are defined on $kG\ot R$, 
the  map in $ (6)$ is defined on $(V\ot R)\cap (R\ot V)\subset V\ot V\ot V$,
and
 $(6)$ implies that the maps in $(4)$ and $(5)$ are
also defined on $(V\ot R)\cap (R\ot V)$. 
\end{thm}

We will prove the theorem in Section~\ref{sec:hom2} as a corollary
of Theorem~\ref{thm:hom}, after first developing 
some homological algebra in Sections~\ref{sec:def} and~\ref{sec:hom}.

The theorem above includes
the case of filtered quadratic algebras
defined over the ring $kG$ instead of the field $k$.
Such algebras preserve the action of $kG$
and 
correspond to the case
$\lambda = 0$ in the theorem above.
We recover a result from~\cite{quad} 
which we rephrase below to highlight the role of the twisting map $\sigma$. 
The theorem was developed to provide tools particularly in the case
that $kG$ is not semsimple.

Note that the action of $G$ on itself
by conjugation induces an action
on the parameter functions $\alpha$ and $\beta$
(with $(^g\alpha)(r)=\, ^g(\alpha(\, ^{g^{-1}}r))$
as usual and $^g(v\otimes h)=\, ^gv \otimes ghg^{-1}$
for $r$ in $R$, $g$ in $G$, and $v$ in $V$).

\begin{thm}\cite[Theorem 5.4]{quad}
Let $G$ be a finite group and let $V$ be a $kG$-module.
Let $S=T_k(V)/(R)$ be a Koszul algebra
with subspace $R\subset V\ot V$ closed under the action of $G$.
Then a filtered quadratic algebra 
$\mathcal{H}$ is a PBW deformation of $S\rtimes G$ preserving 
the action of $G$
if and only if 
\[
\mathcal{H}\cong \mathcal{H}_{0,\alpha,\beta} 
\]
for some $G$-invariant linear parameter functions
$\alpha:R\rightarrow V\otimes kG, \ 
\beta: R\rightarrow  kG$ 
satisfying, upon projection to $S\rtimes G$,
\begin{itemize}
\item[(i)]
  $ \ \alpha\ot 1-1\ot\alpha=0$ ,
\item[(ii)]
$ \ \alpha  ( (1\ot \sigma)(\alpha\ot 1)- 1\ot\alpha) 
=1\ot \beta-\beta\ot 1$,
\item[(iii)]\rule{0ex}{2
ex}
$  \beta ((1\ot \sigma)(\alpha\ot 1) - 1\ot\alpha)= 0$.
\end{itemize}
Here, the map in (i) is defined on $(V\ot R)\cap (R\ot V)$, and
(i) implies that the maps in (ii) and (iii) are also defined
on $(V\ot R)\cap (R\ot V)$. 
\end{thm}

\begin{proof}
The additional hypothesis, that the action of $G$ is 
preserved in the deformation, is equivalent to setting $\lambda =0$
in Theorem~\ref{thm:main}. In this case, 
Condition~(1) of Theorem~\ref{thm:main} is vacuous,
and Conditions~(2) and (3) are equivalent to $G$-invariance of $\alpha$
and $\beta$. 
Conditions (4), (5), (6) become Conditions (ii), (iii), (i) here, 
respectively.
\end{proof}

\begin{remark}
{\em The conditions of the above theorems generalize those of Braverman and 
Gaitsgory~\cite[Lemma~3.3]{BG} 
from Koszul algebras $S$ to skew group algebras $S\rtimes G$.
Their Condition (I) corresponds to our 
Conditions
(1), (2), and (3) in Theorem~\ref{thm:main};
these conditions limit the possible relations of filtered degree~1.
The nonmodular case can be proven using the theory of Koszul rings
over the semisimple ring $kG$, as in \cite{EG}.
In the modular case, when char$(k)$ divides $|G|$,
we found in~\cite{quad} that
more complicated homological information
is required to obtain PBW conditions
using this approach.
}\end{remark} 

\section{Deformations}\label{sec:def}

In this section, we recall the general theory of deformations and Hochschild 
cohomology that we will need and show how it applies to the 
algebras $\chabl$ of Theorem~\ref{thm:main}. 

Recall that for any $k$-algebra $A$, 
the   Hochschild cohomology of an $A$-bimodule $M$ in degree $n$ is
$$
  \HH^n (A,M) = \Ext^n_{A^e}(A,M),
$$
where $A^e=A\ot A^{op}$ is the enveloping algebra of $A$, 
and the bimodule structure of $M$ defines it as an $A^e$-module.
In the case that $M=A$, we abbreviate $\HH^n(A)= \HH^n(A,A)$.

\subsection*{Bar and reduced bar resolutions}
Hochschild cohomology  can be defined using
the {\em bar resolution}, that is, the free resolution
of the $A^e$-module $A$ given as: 
\[
 \cdots \stackrel{\delta_3}{\longrightarrow} A\ot A\ot A\ot A
  \stackrel{\delta_2}{\longrightarrow} A\ot A\ot A
  \stackrel{\delta_1}{\longrightarrow} A\ot A \stackrel{\delta_0}{\longrightarrow}
    A\rightarrow 0 , 
\]
where 
\begin{equation}\label{eqn:delta}
   \delta_n (a_0\ot \cdots\ot a_{n+1}) = \sum_{i=0}^n (-1)^i a_0\ot \cdots
      \ot a_i a_{i+1}\ot \cdots\ot a_{n+1}
\end{equation} 
for all $n\geq 0$ and $a_0,\ldots,a_{n+1}\in A$.
If $A$ is an $\NN$-graded algebra, then each tensor power of $A$ is canonically a
graded $A$-bimodule. The 
Hochschild cohomology of $A$ inherits this grading from
the bar resolution and thus is bigraded: 
$\HH^i(A)=\bigoplus_j \HH^{i,j}(A)$
with 
$\HH^{i,j}(A)$ the subspace consisting of
homogeneous elements of graded degree $j$ as maps. 
For our arguments, we will need to use the {\em reduced bar resolution},
which replaces the $A^e$-module $A^{\ot (n+2)}$, for each $n$, by its vector
space quotient $A\ot (\overline{A})^{\ot n}\ot A$, where
$\overline{A} = A/k$ (the vector space quotient by all scalar multiples of the 
multiplicative identity $1_A$ in $A$). 
The differentials on the bar resolution factor through these quotients
to define differentials for the reduced bar resolution. 

\subsection*{Deformations}
A {\em deformation of $A$ over $k[t]$} is an associative
$k[t]$-algebra $A_t$ with underlying vector space $A[t]$
such that $A_t|_{t=0}\cong A$ as  algebras.
The product $*$ on a deformation $A_t$ of $A$ is determined by its
values on pairs of elements of $A$, 
\begin{equation}\label{star-formula}
    a_1 * a_2 = a_1a_2 + \mu_1(a_1\ot a_2) t + \mu_2(a_1\ot a_2) t^2 + \cdots
\end{equation}
where $a_1a_2$ is the product of $a_1$ and $a_2$ in $A$ and 
each $\mu_j: A\ot A \rightarrow A$ 
is some $k$-linear map (called the {\em $j$-th multiplication map}) 
extended to be linear over $k[t]$.
(We require that only finitely many terms in the above
expansion 
for each pair $a_1, a_2$ are nonzero.)
We may (and do) assume that $1_A$  
is the multiplicative identity with respect to the multiplication
$*$ of $A_t$.
(Each deformation 
is equivalent to one with $1_A$ serving as the multiplicative identity;
see~\cite[p.\ 43]{GerstenhaberSchack83}.)

We identify the  maps $\mu_i$ with 2-cochains
on the reduced bar resolution
using the 
canonical isomorphism 
$
  \Hom_{k}(\overline{A}\ot  \overline{A}, A)\cong 
\Hom_{A^e}(A\ot \overline{A} \ot \overline{A}\ot  A, A).
$
 (Our assumptions
imply that the value of $\mu_i$ is 0 if either argument is the
multiplicative identity of $A$.)
We will use the same notation for elements of $A$ and $\overline{A}$
when no confusion will arise. 

Associativity of the multiplication $*$ implies
certain conditions on the maps $\mu_i$ which
are elegantly phrased in~\cite{Gerstenhaber64} in terms of the differential
$\delta$ and the 
Gerstenhaber bracket $[ \  , \  ]$, as we explain next. 
The {\em Gerstenhaber  bracket} for  2-cochains $\xi, \nu$
on the (reduced) bar resolution 
is the 3-cochain defined by
\begin{equation}\label{eqn:G-brack}
\begin{aligned}
{[} \xi, \nu {]} (a_1\ot a_2\ot a_3) 
\ \ = \ \
\xi(\nu(a_1\ot a_2)\ot a_3)
   -\xi(a_1\ot \nu(a_2\ot a_3)) \\
 \ \ \ + \nu(\xi(a_1\ot a_2)\ot a_3)
  -\nu(a_1\ot \xi(a_2\ot a_3))
\end{aligned}
\end{equation}
for all $a_1, a_2, a_3\in A$.
See~\cite{Gerstenhaber63} for the 
definition of Gerstenhaber bracket in other degrees.

\subsection*{Obstructions}
If $A_t$ is a deformation of a $k$-algebra $A$ over $k[t]$, 
associativity of multiplication $*$ implies in particular  
(see \cite{Gerstenhaber64}) 
that
\begin{align}
\label{obst0}
\delta_3^*(\mu_1)&=0
&\text{ ($\mu_1$ is a Hochschild 2-cocycle),}&\\
\label{obst1}
\delta_3^*(\mu_2)&=\tfrac{1}{2}[\mu_1, \mu_1]
&\text{ (the first obstruction vanishes)},& \text{ and }\\
\label{obst2}
\delta_3^*(\mu_3)&=[\mu_1, \mu_2]
&\text{ (the second obstruction vanishes).}&
\end{align}
Associativity of the multiplication $*$
also implies that higher degree ``obstructions''
vanish, i.e., it forces necessary 
conditions on all the $\mu_j$. 
We will only need to look closely at the above beginning obstructions:
Higher degree obstructions relevant to our setting
will automatically vanish because
of the special nature of Koszul algebras
(see the proof of Theorem~\ref{thm:hom}).

\subsection*{Graded deformations}
Assume  that the $k$-algebra  $A$ is $\NN$-graded.
Extend the grading
on $A$ to $A[t]$ by setting $\deg (t) = 1$.
A {\em graded deformation of $A$ over $k[t]$} 
is a deformation of $A$ over $k[t]$ that is graded, i.e.,  
each map $\mu_j:A\otimes A \rightarrow A$ is homogeneous of degree $-j$.
An {\em $i$-th level graded deformation of $A$} is a deformation
over $k[t]/(t^{i+1})$, i.e., an algebra $A_i$ with underlying
vector space $A[t]/(t^{i+1})$ and multiplication as 
in (\ref{star-formula}) in which terms involving powers of $t$ greater 
than $i$ are 0.
An $i$-th level graded deformation $A_i$ of $A$ {\em lifts} (or {\em extends})
to an $(i+1)$-st level graded deformation $A_{i+1}$
if the $j$-th multiplication maps of $A_i$
and $A_{i+1}$ coincide for all $j\leq i$. 

We next point out that the algebra $\chabl$
defined in (\ref{eqn:Habc})
gives rise to a graded deformation
of $S\rtimes G$ in case it has the PBW property. 


\begin{prop}\label{prop:PBW}
If $\chabl$ is a PBW deformation of $S\rtimes G$,
then $\chabl$ is the fiber of a deformation of 
$A=S\rtimes G$:
$$\chabl\cong A_t\big|_{t=1}$$
as filtered algebras, for $A_t$ a graded
deformation of $S\rtimes G$ over $k[t]$.
\end{prop}
\begin{proof}
We define the algebra $A_t$ by 
$$
A_t=
T_{kG}(kG\ot V\ot kG)[t] / 
(r-\alpha(r)t-\beta(r)t^2, \ r'-\lambda(r')t:
r\in R, r'\in R').
$$
Since $\chabl$ has the PBW property,
$A_t$ and $(S\rtimes G)[t]$
are isomorphic as $k[t]$-modules:
Define a $k$-linear map from $S\rtimes G$ to $T_{kG}(kG\ot V\ot kG)$ so
that composition with the quotient map onto $\chabl$ is an isomorphism
of filtered vector spaces, and extend to a $k[t]$-module homomorphism
from $S\rtimes G [t]$ to $T_{kG}(kG\ot V\ot kG)[t]$.
The composition of this map with the quotient map onto $A_t$ 
can be seen to be an isomorphism of vector spaces by a degree argument.  
The rest of the proof is a straightforward generalization of
\cite[Proposition~6.5]{ueber}, which is the case $S=S(V)$ and $\alpha = 0$. 
Here, $r$ replaces $v\ot w - w\ot v$ and 
the first and second multiplication maps $\mu_1$ and $\mu_2$
satisfy 
\[
\begin{aligned} 
 & \lambda(g\ot v\ot 1 - 1\ot {}^gv\ot g)=\mu_1( g\ot v)-\mu_1(  {}^gv\ot g ),\\
 & \alpha(r) = \mu_1(r), \ \ \ 
   \mbox{ and } \ \ \ \beta(r)=\mu_2( r)
\end{aligned}
\]
for all $g$ in $G$, $v$ in $V$, and $r$ in $R$.
\end{proof}

\section{Hochschild cohomology of group twisted Koszul algebras}\label{sec:hom}
We will look more closely at the Hochschild 2-cocycle condition (\ref{obst0})
and the obstructions (\ref{obst1})
and (\ref{obst2}) in the case that $A$ is a group twisted
Koszul algebra $S\rtimes G$.
A convenient resolution for this purpose was introduced by Guccione, Guccione, 
and Valqui~\cite{GGV}.
We now recall from \cite{quad} a  modified version of this construction.

\subsection*{Twisted product resolution}
Again, let $S$ be a Koszul algebra with finite dimensional
generating
$k$-vector space $V$ and subspace of relations $R\subset V\otimes V$:
$$
S=T_k(V)/(R)\, .$$
Since $S$ is Koszul, the complex
$$\cdots\rightarrow
{K}_3\overset{d_3}{\longrightarrow}
{K}_2\overset{d_2}\longrightarrow
{K}_1  \overset{d_1}\longrightarrow
{K}_0\overset{d_0}{\longrightarrow}
S\rightarrow 0
$$
is a free $S^e$-resolution of $S$, where $K_n = S\ot \tilde{K}_n\ot S$ with 
$\tilde{K}_0=k$, $\tilde{K}_1 = V$, and 
$$\tilde{K}_n=
 \bigcap_{j=0}^{n-2}(V^{\otimes j}\otimes R\otimes V^{\otimes(n-2-j)}),
\quad n\geq 2.
$$ 
Identify $K_0$ with $S\ot S$. 
The differential is restricted from that of the (reduced) bar resolution of $S$,
defined in (\ref{eqn:delta}),
so that $d_n=\delta_n|_{{K}_n}$.

Let $G$ be finite group acting by graded automorphisms on $S$ and 
set $A=S\rtimes G$.
The {\em twisted product resolution} $X_{\DOT}$ 
of $A$ as an $A^e$-module
is  the total complex 
of the double complex $X_{\DOT,\DOT}$,
where
\begin{equation}\label{xij2}
X_{i,j} = A\ot (\overline{kG})^{\ot i} \ot \tilde{K}_j\ot A ,
\end{equation}
and $A^e$ acts by left and right multiplication on the outermost
tensor factors $A$:
$$
\begin{small}
\xymatrixcolsep{9ex}
\xymatrixrowsep{9ex}
\xymatrix{
&
X_{0,3}\ar[d]^{d^v_{0,3}} &
\vdots \ar[d] &
\vdots \ar[d] \\
&
X_{0, 2} \ar[d]^{d^v_{0,2}} &
X_{1, 2} \ar[l]^{d^h_{1, 2}} \ar[d]^{d^v_{1,2}}  &
X_{2, 2} \ar[l]^{d^h_{2, 2}}  \ar[d]^{d^v_{2,2}}  &
\ar[l]\cdots \\
&
X_{0, 1} \ar[d]^{d^v_{0, 1}} &
X_{1, 1} \ar[l]^{d^h_{1,1}}   \ar[d]^{d^v_{1, 1}} &
X_{2, 1} \ar[l]^{d^h_{2,1}} \ar[d]^{d^v_{2, 1}} &
\ar[l] \cdots \\
&
X_{0, 0}  &
X_{1, 0} \ar[l]^{d^h_{1,0}} &
X_{2, 0} \ar[l]^{d^h_{2,0}} &
\ar[l] \cdots
}
\end{small}
$$
To define the differentials, we first identify each $X_{i,j}$ with a
tensor product over $A$ (see~\cite[Section~4]{quad}),
\begin{equation}\label{eqn:xij1}
  X_{i,j}\cong  (A\ot (\overline{kG})^{\ot i}\ot kG)\ot _A (S\ot\tilde{K_j}\ot A),
\end{equation}
where the right action of $A$ on $A\ot (\overline{kG})^{\ot i}\ot kG$ is given by
\[
   (a\ot g_1\ot \cdots \ot g_i\ot g_{i+1}) sh =
   a ( {}^{g_1\cdots g_{i+1}} s ) \ot g_1\ot \cdots
   \ot g_i\ot g_{i+1} h 
\]
and
the left action of $A$ on $S\ot\tilde{K}_j\ot A$ is given by
\[
   sh (s'\ot x\ot a) = s ( {}^h s')\ot {}^hx \ot ha 
\]
for all $g_1,\ldots, g_{i+1}, h$ in $G$, 
$s,s'$ in $S$, and $a$ in $A$.
(We have suppressed tensor symbols in writing elements
of $A$ to avoid confusion with tensor
products defining the resolution.)
The horizontal and vertical differentials 
on the bicomplex $X_{\DOT, \DOT}$, given as a tensor product
over $A$ via (\ref{eqn:xij1}), are then defined by
$
d_{i,j}^h=d_i\ot 1
$
and
$d^v_{i,j}=(-1)^i \ot d_j$,
respectively, where the notation $d$ is used for both the differential
on the reduced bar resolution of $kG$ (induced to an $A\ot (kG)^{op}$-resolution) 
and on the Koszul resolution of $S$ (induced to an $S\ot A^{op}$-resolution).
Setting $X_n = \oplus_{i+j=n} X_{i,j}$ for each $n\geq 0$
yields the total complex $X_{\DOT}$: 
\begin{equation}\label{resolution-X}
  \cdots\rightarrow X_2\rightarrow X_1\rightarrow X_0\rightarrow A\rightarrow 0,
\end{equation}
 with differential in positive degrees $n$ given by 
$d_n= \sum_{i+j=n} (d_{i,j}^h + d_{i,j}^v)$,
and in degree 0 by the multiplication map. 
By~\cite[Theorem~4.3]{quad},  $X_{\DOT}$ is a free resolution
of the $A^e$-module $A=S\rtimes G$. 

\subsection*{Chain maps between reduced bar 
and twisted product resolutions}
We found in~\cite{quad} useful chain maps
converting between the bar resolution 
and the (nonreduced) twisted product resolution $X_{\DOT}$
of $A=S\rtimes G$.  
We next extend~\cite[Lemma~4.7]{quad},
adding more details and  adapting it for
use with the
reduced bar resolution.  See also~\cite[Lemma~7.3]{ueber}
for the special case $S=S(V)$. 
We consider elements of $\tilde{K}_j
\subset  V^{\otimes j} $
to have graded degree $j$ and elements of $(\overline{kG})^{\ot i}$ to have graded degree 0.

\begin{lemma}\label{lem:maps}
For $A=S\rtimes G$, there exist $A$-bimodule homomorphisms 
$\phi_n: X_n\rightarrow A\ot \overline{A}^{\,\ot n }\ot A$
and $\psi_n: A\ot \overline{A}^{\,\ot n }\ot A \rightarrow X_{n}$ 
such that the diagram 
$$
\begin{small}
\xymatrix{
\cdots\ar[r] &
X_{2} \ar[r]^{d_2} \ar@<-.5ex>[d]_{\phi_2}&
X_{1} \ar[r]^{d_1} \ar@<-.5ex>[d]_{\phi_1}&
X_{0} \ar[r]^{d_0}\ar@<-.5ex>[d]_{\phi_0} &
A \ar[r] \ar@<-.5ex>[d]& 0 \\
\cdots\ar[r] &
A\ot \overline{A}\ot \overline{A} \ot A \ar[r]^{\delta_2} \ar@<-.5ex>[u]_{\psi_2}&
A\ot \overline{A} \ot A \ar[r]^{\delta_1}\ \ar@<-.5ex>[u]_{\psi_1}&
A\ot A \ar[r]^{\delta_0} \ar@<-.5ex>[u]_{\psi_0}&
A \ar[r] \ar@<-.5ex>[u]_{=} & 0& \\
}
\end{small}
$$
commutes, the maps $\phi_n$, $\psi_n$ are of graded degree~0, and $\psi_n\phi_n$
is the identity map on $X_n$ for $n=0,1,2$. 
\end{lemma}

In fact, it can be shown that there are chain maps  such that 
$\psi_n\phi_n$ is the identity map on $X_n$
for each $n$.
We will not need this more general statement here, but
rather some of the explicit values of the maps in low degrees as given in
the proof of the lemma. 

\begin{proof}
We again suppress tensor symbols in writing elements
of $A$ to avoid confusion with tensor
products defining the resolution.
In degree 0, $\psi_0$ and $\phi_0$ 
may be chosen to be identity
maps on $A\ot A$.
As in \cite[Lemma~4.7]{quad}, we may set
\[
 \phi_1(1\ot g\ot 1)=1\ot g\ot 1 \ \ (\mbox{on }X_{1,0}), 
\quad \ \phi_1(1\ot v\ot 1) = 1\ot v\ot 1 \ \ (\mbox{on }X_{0,1}) ,
\]
for all nonidentity $g$ in $G$ and $v$ in $V$, 
and these values determine $\phi_1$ as an $A$-bimodule map.
Moreover, we set
\[
\begin{aligned}
\psi_1(1\ot g\ot 1) & = 1\ot g\ot 1 \ \ \ (\mbox{in } X_{1,0}),\\
\psi_1(1\ot vg\ot 1) &=1\ot v\ot g + v\ot g\ot 1 \ \ 
  \  (\mbox{in } X_{0,1}\oplus X_{1,0})
\end{aligned}
\]
for nonidentity $g$ in $G$ and $v$ in $V$ and 
use the identification (\ref{eqn:xij1}) for evaluating the differential
to check that $d_1\psi_1 = \psi_0\delta_1$ on these arguments. 
In order to extend $\psi_1$ to an $A$-bimodule map
on $A\ot \overline{A}\ot A$, 
we first choose a homogeneous 
vector space basis of 
$\overline{A}$ 
consisting of
the elements $g\neq 1_G$, $vg$,
and $sg$
as $g$ ranges
over the elements of $G$, 
$v$ ranges over a $k$-basis of $V$,
and $s$ ranges over a $k$-basis
of homogeneous elements of $S$ of degree $\geq 2$.
Tensoring each of these elements
on the left and right by $1$ then
gives a free $A$-bimodule basis of $A\ot\overline{A}\ot A$.
The function $\psi_1$ is already defined
on elements of the form $1\ot g\ot 1$ and
$1\ot vg\ot 1$;
we may define $\psi_1$ on elements of the form $1\ot s\ot 1$ so that
$d_1\psi_1= \psi_0 \delta_1$ on these elements
and then define
\[
\psi_1(1\ot sg\ot 1) = \psi_1(1\ot s\ot g) + s\ot g \ot 1
\] 
so that 
$d_1\psi_1= \psi_0 \delta_1$ on these elements as well. 
Then $\psi_1\phi_1$ is the identity map on $X_1$,
by construction. 

Define $\phi_2$ by setting
\[
\begin{aligned}
& \phi_2(1\ot g\ot h\ot 1) & &=& &1\ot g\ot h\ot 1& 
\ \ \ \ &(\mbox{on } X_{2,0}), \\
& \phi_2(1\ot g\ot v\ot 1) & &=& &1\ot g\ot v\ot 1 - 1\ot {}^gv\ot g\ot 1&
  \ \ \ \ &(\mbox{on }X_{1,1}),\\
 & \phi_2(1\ot r\ot 1) & &=& &1\ot r\ot 1& \ \ \ \ 
&(\mbox{on } X_{0,2})
\end{aligned}
\]
for all nonidentity $g,h$ in $G$, $v$ in $V$, and $r$ in $R$. 
One may check that $\delta_2\phi_2 = \phi_1 d_2$.
Now set
\[
\begin{aligned}
  &\psi_2(1\ot g\ot h\ot 1)& &=& &1\ot g\ot h\ot 1& 
\ \ \ \ \ \ \ \ &(\mbox{in }X_{2,0}),\\
   &\psi_2(1\ot vh\ot g\ot 1)& &=& &v\ot h\ot g\ot 1& 
&(\mbox{in } X_{2,0}),\\
   &\psi_2(1\ot g\ot vh\ot 1)& &=& &1\ot g\ot v \ot h + {}^gv\ot g\ot h\ot 1& 
&(\mbox{in }X_{1,1}\oplus X_{2,0}),\\
  &\psi_2(1\ot r\ot 1)& &=& &1\ot r\ot 1& 
&(\mbox{in }X_{0,2}),\\
\end{aligned}
\]
for all  $g,h$ in $G$, $v$ in $V$, and
$r$ in $R$. 
A calculation shows that $d_2\psi_2=\psi_1\delta_2$ on these elements.
Letting $g,h$ range over the elements of $G$, $v$ over
a $k$-basis of $V$, and $r$ over a $k$-basis of $R$, we 
obtain a linearly independent set consisting of elements of the form
$1\ot g\ot h\ot 1$, 
$1\ot vh\ot g\ot 1$,
$1\ot g\ot vh\ot 1$, 
and $1\ot r\ot 1$
on which $\psi_2$ has already been defined.
Extend the $k$-basis of $R$ to a $k$-basis of $V\ot V$ by including
additional elements of the form $v\ot w$ for $v,w$ in $V$. Now 
define $\psi_2(1\ot v\ot w\ot 1)$ 
arbitrarily subject to the condition that $d_2\psi_2 = \psi_1\delta_2$
on these elements. 
Let 
\begin{equation}\label{explicitpsivalues}
\begin{aligned}
\psi_2(1\ot vg\ot w\ot 1)\ &=& &\psi_2(1\ot v\ot {}^gw\ot g) + v\ot g\ot w\ot 1,& 
\ \text{ and }\\
\psi_2(1\ot v\ot wg\ot 1)\ &=& &\psi_2(1\ot v\ot w\ot g)&
\end{aligned}
\end{equation}
for all $g$ in $G$ and $v,w$ in $V$.
One checks that $d_2\psi_2 = \psi_1\delta_2$ on these elements as well. 
Extend these elements to a free $A$-bimodule basis of $A\ot\overline{A}
\ot\overline{A}\ot A$.
We may 
define $\psi_2$ on the remaining free basis elements so that 
$d_2\psi_2 =\psi_1\delta_2$. By construction, $\psi_2\phi_2$ is the
identity map on $X_2$.
\end{proof}

We will  need some further values of $\phi$ in homological degree 3,
which we set in the next lemma.
The lemma is proven by directly checking the chain map condition. 
Other values of $\phi_3$ may be defined by 
extending to a free $A$-bimodule basis of $A\ot (\overline{A})^{\ot 3}\ot A$.  
\begin{lemma}
\label{eqn:phi3}
We may choose the map
$\phi_3$  in Lemma~\ref{lem:maps}
so that
$$
\begin{aligned}
\phi_3(1\ot x\ot 1) & = 1\ot x\ot 1 \ \ \ (\mbox{on }X_{0,3}),\\
  \phi_3(1\ot g\ot r\ot 1) 
 & =   1\ot g\ot r\ot 1 - (1\ot \sigma\ot 1\ot 1)(1\ot g\ot r\ot 1)\\
 &\hspace{.4cm} + (1\ot 1\ot\sigma\ot 1)(1\ot \sigma\ot 1\ot 1)(1\ot g\ot r
   \ot 1) \ \ \ (\mbox{on }X_{1,2}) 
\end{aligned}
$$
for all nonidentity $g$ in $G$, $r$ in $R$, 
and $x$ in $(V\ot R)\cap (R\ot V)$. 
\end{lemma}

\section{Homological PBW conditions}\label{sec:hom2}
We now give homological conditions for a filtered algebra
to be a PBW deformation of a Koszul algebra twisted
by a group.
These conditions are a translation of the necessary
homological Conditions~(\ref{obst0}), 
(\ref{obst1}), and (\ref{obst2}) into conditions on the parameter
functions $\alpha,\beta, \lambda$ defining a potential deformation; we prove these conditions
are in fact sufficient.

Again, let $S$ be a Koszul algebra generated by a finite dimensional
vector space $V$ with defining relations 
$R$ and an action of a finite group $G$ by graded automorphisms.
Let $R'$ be the space of group action relations defined in
(\ref{eqn:R-prime}).  Let $A=S\rtimes G$.
We use the resolution $X_{\DOT}$ of (\ref{xij2}) to express
the Hochschild cohomology $\HH^{\DOT}(A)$. 

\begin{remark}\label{rk:extend-fcns}
{\em 
Just as in \cite[Lemma 8.2]{ueber}, we may 
identify the $k$-linear functions 
$$\alpha : R\rightarrow V\ot kG, \ \ 
\beta: R\rightarrow kG , \ \ \mbox{ and } \lambda: R'\rightarrow kG$$ 
with
2-cochains on the resolution $X_{\DOT}$, i.e.,
$A$-bimodule homomorphisms from $X_2$ to $A$. 
Indeed, both $\alpha$ and $\beta$ 
extend uniquely to 
cochains on $X_{0,2}= A\ot R\ot A$
since a cochain is an $A$-bimodule homomorphism 
and thus determined there by its values on $R$.
Similarly, 
$\lambda$ corresponds to a unique
cochain on $X_{1,1}$
taking the value $\lambda (g\ot v\ot 1-  1\ot {}^gv\ot g)$
on elements of the form $1\ot g \ot v\ot 1$.
Here we identify the target spaces of $\alpha,\beta,\lambda$ 
with subspaces of $A$. 
We  extend these cochains 
defined by $\alpha, \beta,\lambda$ to all of $X_{\DOT}$
by setting them to be 0 on 
the components of $X_{\DOT}$ 
on which we did not already define them.  
}\end{remark}

We fix  choices of chain maps $\phi$, $\psi$ satisfying Lemmas~\ref{lem:maps} 
and~\ref{eqn:phi3}.
We define the Gerstenhaber bracket of cochains on $X_{\DOT}$
by transferring 
the Gerstenhaber bracket~(\ref{eqn:G-brack}) on the (reduced) bar resolution 
to $X_{\DOT}$ using these chain maps: 
If $\xi,\nu$ are Hochschild cochains on $X_{\DOT}$, we define
\begin{equation}\label{eqn:K-brack}
  [\xi,\nu] = \phi^* ( [\psi^*(\xi), \psi^*(\nu)]) ,  
\end{equation}
another cochain on $X_{\DOT}$.
At the chain level, this bracket depends on the choice of chain maps
$\phi,\psi$, although at the level of cohomology, it does not. 
The choices we have made in Lemmas~\ref{lem:maps} and~\ref{eqn:phi3}
 provide valuable information, as we see next.

\begin{thm}\label{thm:hom}
Let $S$ be a Koszul algebra over the field $k$ generated by
a finite dimensional vector space $V$.
Let $G$ be a finite group acting on $S$ by
graded automorphisms. 
The algebra $\chabl$ defined in (\ref{eqn:Habc}) is a PBW deformation of $S\rtimes G$
if and only if 
\begin{itemize}
\item[(a)]
$d^*(\alpha+\lambda)=0$, 
\item[(b)]
$[\alpha+\lambda, \alpha+\lambda]=2d^*\beta$, and
\item[(c)]
$[\lambda+\alpha, \beta]=0$,
\end{itemize}
where $\alpha,\beta,\lambda$ are identified with 
cochains on the
twisted product resolution $X_{\DOT}$ as in Remark~\ref{rk:extend-fcns}. 
\end{thm}
\begin{proof}
We adapt ideas of~\cite[Theorem 4.1]{BG}, first
 translating the above Conditions~(a), (b), and~(c)  
to conditions 
on the reduced bar resolution itself. 
The proof is similar to that of~\cite[Theorem 5.4]{quad},
but certain arguments must be
altered to allow for the additional parameter function $\lambda$.

If $\chabl$ is a PBW deformation of $S\rtimes G$, then by Proposition~\ref{prop:PBW},
there are Hochschild 2-cochains $\mu_1$ and $\mu_2$ on the
(reduced) bar resolution such that 
the Conditions~(\ref{obst0}),
(\ref{obst1}), and (\ref{obst2}) hold, that is, 
$\mu_1$ is a Hochschild 2-cocycle, $[\mu_1,\mu_1]=2\delta^*(\mu_2)$,
and $[\mu_1,\mu_2]$ is a coboundary.
By the proofs of Proposition~\ref{prop:PBW} and Lemma~\ref{lem:maps}, 
$$
\alpha+\lambda =\phi_2^*(\mu_1)  \ \ \mbox{ and } \ \ 
\beta =\phi_2^*(\mu_2) . 
$$
Since $\mu_1$ is a cocycle, it follows that $d^*(\alpha +\lambda)=0$,
that is, Condition (a) holds. 
For Condition (b), note that each side of the equation is automatically 0
on $X_{3,0}$ and on $X_{2,1}$, by a degree argument.
We will evaluate each side of the equation on $X_{1,2}$ and on $X_{0,3}$. 
By definition,
\begin{eqnarray*}
 [\alpha + \lambda, \alpha+\lambda] & = & \phi^*[\psi^*(\alpha+\lambda), \psi^*(\alpha + \lambda) ] \\
    &=& \phi^*([\psi^*\phi^*(\mu_1),\psi^*\phi^*(\mu_1)])\\
   & = & 2 \phi^*(\psi^*\phi^*(\mu_1) (\psi^*\phi^*(\mu_1)\ot 1 - 1\ot \psi^*\phi^*(\mu_1)).
\end{eqnarray*}
We evaluate on $X_{1,2}$. By Lemma~\ref{eqn:phi3}, the image of $\phi_3$ on $X_{1,2}$
is contained in 
\[
   (kG\ot \Ima (\phi_2\big|_{X_{0,2}}))\oplus 
(V\ot \Ima(\phi_2\big|_{X_{1,1}}))\cap 
 (\Ima (\phi_2\big|_{X_{0,2}})\ot kG)\oplus 
(\Ima(\phi_2\big|_{X_{1,1}})\ot V) .
\]
Therefore, since $\psi_2\phi_2$ is the identity map, applying $\psi^*\phi^*(\mu_1)\ot 1 - 1\ot \psi^*\phi^*(\mu_1)$
to an element in the image of $\phi_3$ is 
the same as applying $\mu_1\ot 1 - 1\ot \mu_1$.
Since $\mu_1$ is a Hochschild 2-cocycle, the image of $\mu_1\ot 1 - 1\ot \mu_1$ is 0 upon projection to $S\rtimes G$,
which implies that the image of $\mu_1\ot 1 - 1\ot \mu_1$ on $\phi_3(X_{1,2})$ is
contained in the subspace of $\overline{A}\ot \overline{A}$ spanned by all $g\ot v - {}^g v\ot g$
for nonidentity $g$ in $G$ and $v$ in $V$.
This is in the image of $\phi_1$, and so again, applying $\psi^*\phi^*(\mu_1)$ is the same as applying $\mu_1$.
Hence $[\alpha+\lambda,\alpha+\lambda] = \phi^*([\mu_1,\mu_1])$ on $X_{1,2}$. 
Condition~(\ref{obst1})
then implies that Condition (b) holds on $X_{1,2}$. 
A similar argument verifies 
Condition~(b) on
 $X_{0,3}$.
Condition (c) holds by a degree argument:  
The bracket $[\lambda+\alpha, \beta]$ is cohomologous to $[\mu_1,\mu_2]$, which 
by (\ref{obst2}) is a coboundary. So $[\alpha+\lambda,\beta]$ is itself a coboundary:  
 $[\lambda+\alpha, \beta]=d^*(\xi)$ for some 2-cochain $\xi$.
Now $[\lambda +\alpha,\beta]$ is of graded degree~$-3$, and 
the only 2-cochain on $X_{\DOT}$ of graded degree~$-3$ is 0. 

For the converse, assume Conditions (a), (b), and (c) hold.
We may now set
$\mu_1=\psi^*(\alpha + \lambda)$ and $\mu_2=\psi^*(\beta)$. 
Set 
\[
\gamma =\delta_3^*(\mu_2)-\tfrac{1}{2}[\mu_1,\mu_1] .
\]
Condition~(a) of the theorem implies that $\alpha+\lambda$
is a 2-cocycle and thus $\mu_1$ is a  2-cocycle
on the reduced bar resolution of $A$.
The 2-cocycle $\mu_1$ then is a first multiplication map on $A$ and
defines a first level deformation $A_1$
of $A=S\rtimes G$.

Next we will see that Condition~(b) implies 
this first level deformation can be extended to a second 
level deformation.
By Lemma~\ref{lem:maps},
$$
\begin{aligned}
\phi^*_3(\gamma) &= \phi^*_3 (\delta_3^*\left(\psi_2^*(\beta)\right))
-\tfrac{1}{2}\phi_3^*[\psi^*_2(\alpha+\lambda), \psi^*_2(\alpha+\lambda)]\\
&= d^*(\phi^*_2 \psi_2^*(\beta))
-\tfrac{1}{2}[\alpha+\lambda, \alpha+\lambda] \\
&= d_3^*(\beta)
-\tfrac{1}{2}[\alpha+\lambda, \alpha+\lambda] . \\
\end{aligned}
$$
Hence $\phi_3^*(\gamma)=0$ by Condition~(b).
This forces $\gamma$ to be a coboundary, say 
$\gamma=\delta^*(\mu)$ for some 2-cochain $\mu$ on the reduced bar resolution,
necessarily of graded degree $-2$.
Now, 
$$
d^*(\phi^*(\mu))= \phi^* (\delta^*( \mu))=\phi^*(\gamma) =0\, ,
$$
so $\phi^*(\mu)$ is a 2-cocycle. 
Then there must be a 2-cocycle $\mu'$ on the reduced bar resolution 
with $\phi^*\mu'=\phi^*\mu$.
We replace $\mu_2$ by $\tilde{\mu}_2=\mu_2-\mu+\mu'$ so that
$\phi^*(\tilde{\mu}_2)=\beta$
but 
$$2\delta^*(\tilde{\mu}_2)= 2\delta^*(\mu_2+\mu')-2\gamma
= [\mu_1,\mu_1] $$
by the definition of $\gamma$, 
since $\mu'$ is a cocycle. 
Thus the obstruction to lifting $A_1$ to a second level deformation
using the multiplication map $\tilde{\mu}_2$ vanishes, and 
$\mu_1$ and $\tilde{\mu}_2$ together define
a second level deformation $A_2$ of $A$.

We now argue that Condition~(c) implies  $A_2$ lifts
to a third level deformation of $A$.
Adding the coboundary $\mu'-\mu$ to $\mu_2$ adds a coboundary to
$[\mu_2, \mu_1]$, and hence
$[\tilde{\mu}_2, \mu_1] = \delta^*_3(\mu_3)$
for some cochain $\mu_3$ on the reduced bar resolution of graded degree $-3$.
Thus the obstruction to lifting $A_2$ to a third level deformation
vanishes and the multiplication maps
$\mu_1, \tilde{\mu}_2, \mu_3$ define a third level deformation $A_3$ of $A$.

The obstruction to lifting $A_3$ to a fourth level deformation
of $A$ lies in $\HH^{3,-4}(A)$ by~\cite[Proposition 1.5]{BG}.
Applying the map $\phi^*$ to this obstruction gives
a cochain of graded degree $-4$ on $X_3$,
as $\phi$ is of graded degree $0$ as a chain map by Lemma~\ref{lem:maps}.   
But $X_3$ is generated, as an $A$-bimodule, by elements of graded degree 3
or less,
and thus $\phi^*$ applied to 
the obstruction is 0, implying that the obstruction
is a coboundary.  Thus the deformation $A_3$ lifts
to a fourth level deformation $A_4$ of $A$.  Similarly, 
the obstruction to lifting an $i$-th level deformation $A_i$ of $A$
lies in $\HH^{3, -(i+1)}$, and again since $S$ is Koszul, 
the obstruction is a coboundary. 
So the deformation $A_i$ lifts to $A_{i+1}$, 
an $(i+1)$-st level deformation of $A$, for all $i\geq 1$.

The corresponding graded deformation 
$A_t$ of $A$ 
is the vector
space $A[t]$ with multiplication determined by
$$
a*a'=aa'+\mu_1(a,a')t+\mu_2(a, a')t^2 + \mu_3(a, a')t^3+\ldots
$$
for all $a,a'\in A$. 

We next explain that $\chabl$ is isomorphic, as a filtered algebra,
to the fiber $A_t|_{t=1}$. 
First note that $A_t|_{t=1}$ 
is generated by $V$ and $G$
(since 
the associated graded algebra of $A_t$ is $A$).  
Thus we may define an algebra homomorphism 
$$T_{kG}(kG \ot V \ot kG) \longrightarrow A_t\big|_{t=1}$$
and then
use Lemma~\ref{lem:maps} to verify that the elements 
$$
\begin{aligned}
&r-\alpha(r)-\beta(r) \quad &  &\text{ for }r\in R, \text{ and }\\
&g\ot v\ot 1- 1\ot ^gv \ot g-\lambda(g \ot v) \quad & &\text{ for } g\in G, v\in V
\end{aligned}
$$
lie in the kernel.
We obtain a surjective homomorphism of filtered algebras,
$$\chabl\longrightarrow A_t\big|_{t=1}\, . $$
We consider the dimension over $k$ of 
each of the filtered components
in the domain and range:
Each filtered component of
$\chabl$ has dimension at most that of the corresponding
filtered component of  $S\rtimes G$
since its associated graded algebra
is necessarily a quotient of $S\rtimes G$.
But the associated graded algebra of $A_t|_{t=1}$
is precisely $S\rtimes G$, and so
$$
\dim_k(F^d(S\rtimes G))
\geq \dim_k(F^d(\chabl))
\geq \dim_k(F^d(A_t\big|_{t=1}))
= \dim_k(F^d(S\rtimes G)),
$$
where $F^d$ indicates the summand of filtered degree $d$
in $\NN$.
Thus these dimensions are all equal. It follows that 
$\chabl\cong A_t\big|_{t=1}$, and $\chabl$ is a PBW deformation.
\end{proof}

We now prove Theorem~\ref{thm:main} as a consequence of Theorem~\ref{thm:hom},
translating the homological conditions into Braverman-Gaitsgory style
conditions. 

\begin{proof}[Proof of Theorem~\ref{thm:main}]
We explained in Section~\ref{sec:PBW} that each PBW deformation of $S\rtimes G$
has the form $\chabl$ as defined in (\ref{eqn:Habc}) for some parameter
functions $\alpha,\beta,\lambda$. 
Theorem~\ref{thm:hom} gives necessary and sufficient conditions for
such an algebra $\chabl$ to be a PBW deformation of $S\rtimes G$.
We will show that the  Conditions (a), (b), and (c)
 of Theorem~\ref{thm:hom} are equivalent
to those of Theorem~\ref{thm:main}.

When convenient, we identify 
$$\Hom_{A^e}(A\ot \overline{A}^{\, n}\ot A,A)
\cong
\Hom_{k}(\overline{A}^{\, n}, A) \, .
$$
\subsection*{Condition (a): $d^*(\alpha +\lambda) = 0$}
The cochain $d^*(\alpha+\lambda)$ has homological degree~3 and is the
zero function 
if and only if it is 0 on each of $X_{3,0}$, $X_{2,1}$, $X_{1,2}$, and 
$X_{0,3}$. It is automatically 0 on $X_{3,0}$ since $d(X_{3,0})$
trivially intersects $X_{1,1}\oplus X_{0,2}$ on which $\alpha 
+\lambda$ is defined. 

On $X_{2,1}$, 
$d^*(\alpha)= 0$
automatically, as $\alpha$ is 0 on $X_{2,0}\oplus X_{1,1}$.
We evaluate $d^*(\lambda)$ 
on the elements of a free $A^e$-basis of $X_{2,1}$, using the 
identification (\ref{eqn:xij1}) for evaluating the differential: 
\[
\begin{aligned}
d^*(\lambda) &(1\ot g\ot h\ot v\ot 1) \\
 & = \lambda(g\ot h\ot v\ot 1 - 1\ot gh\ot v\ot 1
   + 1\ot g\ot {}^hv\ot h \\
   &\quad\quad + {}^{gh}v\ot g\ot h\ot 1
   - 1\ot g\ot h\ot v ) \\
  &= g \lambda(h\ot v)-\lambda(gh\ot v)+ \lambda(g\ot {}^h v) h
\end{aligned}
\]
in $A$ for all $g,h$ in $G$ and $v$ in $V$,
which can be rewritten as Theorem~\ref{thm:main}(1).
Therefore $d^*(\alpha + \lambda) |_{X_{2,1}} = 0$ 
if and only if
Theorem~\ref{thm:main}(1) holds.
(If $g$ or $h$ is the identity group element $1_G$, then in the evaluation above,
some of the terms are 0 as we are working with the reduced bar resolution.
The condition remains the same in these cases, and merely corresponds to the
condition $\lambda(1_G \ot v) =0$ for all $v$ in $V$.) 

On $X_{1,2}$, $d^*(\alpha+\lambda)=0$
if and only if
\[
\begin{aligned}
d^*(\alpha +\lambda)& (1\ot g\ot r\ot 1)\\
  & = (\alpha + \lambda) ( g\ot r \ot 1 - 1\ot {}^gr \ot g 
   - (\sigma\ot 1\ot 1)(g\ot r\ot 1) - 1\ot g\ot r ) \\
  &=    g \alpha(r) - \alpha( {}^gr)g 
- (1\ot \lambda)(\sigma\ot 1)(g\ot r) - (\lambda\ot 1)(g\ot r) \, 
\end{aligned}
\]
vanishes for all $g$ in $G$ and $r$ in $R$.
(Note that the multiplication map takes $r$
to 0 in $A$.)
This is equivalent to the equality 
\[
   1\ot \alpha - (\alpha \ot 1) (1\ot \sigma)(\sigma\ot 1)
   = (1\ot \lambda)(\sigma\ot 1) + \lambda\ot 1 
\]
as functions on $kG\ot R$ with values in $A$. 
Thus $d^*(\alpha + \lambda)|_{X_{1,2}} = 0$ if and only if
Theorem~\ref{thm:main}(3) holds.

On $X_{0,3}$, $d^*(\lambda)$ is automatically 0 since $\lambda$
is 0 on $X_{0,2}$. So we compute $d^*(\alpha)|_{X_{0,3}}$.
Consider $x$ in $(R\ot V)\cap (V\ot R)$. Then
\begin{equation}\label{image}
 d^*(\alpha) (1\ot x\ot 1)
 =  \alpha (x\ot 1 - 1\ot x)
=  (1\ot \alpha-\alpha\ot 1)(x).
\end{equation}
So $d^*(\alpha + \lambda)|_{X_{0,3}}= 0$ if and only if
$ 1\ot \alpha - \alpha\ot 1$ has image 0 in $A$, i.e., 
Theorem~\ref{thm:main}(6) holds.

\subsection*{Condition (b): $[\alpha+\lambda, \alpha + \lambda] =
2d^*(\beta)$}
On $X_{3,0}$ and on $X_{2,1}$, both sides of this equation
are automatically 0, as their graded degree is $-2$. 
We will compute their values on $X_{1,2}$ and on $X_{0,3}$.
First note that since $\lambda$ and $\alpha$ each have homological
degree 2, by the definition~(\ref{eqn:G-brack}) of bracket, $[\alpha,\lambda]
=[\lambda,\alpha]$ and so
\[
  [\alpha + \lambda, \alpha +\lambda]
   =[\alpha,\alpha]+2[\alpha,\lambda]+[\lambda,\lambda].
\]
We will compute $[\lambda,\lambda]$, $[\alpha,\lambda]$,
and $[\alpha,\alpha]$.

Note that $[\lambda,\lambda]$ can take nonzero values only on
$X_{1,2}$. We will compute its values on elements of the form
$1\ot g\ot r\ot 1$ for $g$ in $G$ and $r$ in $R$. 
By (\ref{eqn:G-brack}), (\ref{eqn:K-brack}), and  Lemmas~\ref{lem:maps} and~\ref{eqn:phi3}, 
$[\lambda , \lambda](1\ot g\ot r\ot 1)
=2\lambda(\lambda\ot 1)(g\ot r)$.
Similarly, 
\[
  [\alpha,\lambda ] (1\ot g\ot r\ot 1) 
     = - \lambda (1\ot \alpha) (g\ot r).
\]
Finally, note that $[\alpha,\alpha]|_{X_{1,2}}= 0$ automatically
for degree reasons.
Just as in our earlier calculation, we find that
\[
   d^*(\beta)(1\ot g\ot r\ot 1) =
  (1\ot \beta - (\beta\ot 1)(1\ot \sigma)(\sigma\ot 1))
(g\ot r).
\]
Therefore, $[\alpha+\lambda,\alpha+\lambda] = 2d^*(\beta)$
on $X_{1,2}$ if and only if
\[
   2\lambda (\lambda\ot 1) - 2 \lambda (1\ot \alpha) =
   2\ot \beta - 2(\beta\ot 1)(1\ot \sigma)(\sigma\ot 1)
\]
on $kG\ot R$. 
This is equivalent to Theorem~\ref{thm:main}(2).

On $X_{0,3}$, the bracket $[\lambda,\lambda]$ vanishes. 
We compute $[\alpha,\lambda]$ and $[\alpha,\alpha]$
on an element $1\ot  x\ot 1$ of $X_{0,3}$
with $x$ in $(V\ot R)\cap (R\ot V)$.
Note that $\psi^*(\alpha)(r)=\alpha\psi(r)=
\alpha(\psi\phi)r=\alpha(r)$ for all $r$ in $R$.
Thus
$$
(\psi^*(\alpha)\ot 1-1\ot \psi^*(\alpha))(x)
=
(\alpha\ot 1-1\ot\alpha)(x)
$$
and therefore 
\begin{eqnarray*}
[ \alpha , \alpha ](1\ot x\ot 1) & = & 2\psi^*(\alpha)(\alpha\ot 1
      - 1\ot \alpha)(x)  \ \mbox{ and}\\
 {[} \alpha , \lambda {]} (1\ot x\ot 1) & = & \psi^*(\lambda)(\alpha\ot 1 -
    1\ot \alpha)(x).
\end{eqnarray*}

We apply $\psi$ to $(\alpha\ot 1 - 1\ot \alpha)(x)$
using Lemma~\ref{lem:maps}.
Since $(\alpha\ot 1)(x)$ lies in $(V\ot kG)\ot V \subset A\ot A$
and $(1\ot \alpha)(x)$ lies in $V\ot (V\ot kG)\subset A\ot A$,
we use~(\ref{explicitpsivalues}) to apply $\psi$:
$$
\psi(\alpha\ot 1 - 1\ot \alpha)(x)
=
(\psi(1\ot \sigma)(\alpha\ot 1) - \psi(1\ot \alpha))(x) \ + \ y
$$
for some element $y$ in $X_{1,1}$. However, $\alpha$ is zero on $X_{1,1}$,
so
$$
\psi^*(\alpha)(\alpha\ot 1 - 1\ot \alpha)(x)
=
\alpha\psi((1\ot \sigma)(\alpha\ot 1) - 1\ot \alpha)(x) \, .
$$
We assume Condition (a) which we have shown implies Condition~(6) of
Theorem~\ref{thm:main}, i.e.,
$((1\ot \sigma)(\alpha\ot 1) - 1\ot \alpha)(x)$
lies in $R\ot kG$ since it is zero upon projection to $A$.
By the proof of Lemma~\ref{lem:maps},
$$\phi((1\ot \sigma)(\alpha\ot 1) - 1\ot \alpha)(x)
=
((1\ot \sigma)(\alpha\ot 1) - 1\ot \alpha)(x),
$$
and applying $\alpha\psi$ gives
$\alpha((1\ot \sigma)(\alpha\ot 1) - 1\ot \alpha)(x)$
since 
$\psi\phi=1$.
Hence 
$$
  [ \alpha, \alpha ](1\ot x\ot 1)  =  
     2\alpha ((1\ot \sigma)(\alpha\ot 1) - 1\ot \alpha) (x)\, .
$$
Similarly, we apply $\psi^*(\lambda)$ to $(\alpha\ot 1 - 1\ot \alpha)(x)$
again using~(\ref{explicitpsivalues}).
Recall that $\lambda$ is only nonzero on $X_{1,1}$,
and $\psi(1\ot\alpha)$ intersects $X_{1,1}$ at $0$;
hence $\psi^*(\lambda)(\alpha\ot 1-1\ot\alpha)(x) =
\psi^*(\lambda)(\alpha\ot 1)(x)$
and
$$
[\alpha,\lambda](1\ot x\ot 1)
=
\psi^*(\lambda)(\alpha\ot 1)(x)=
(\sum_{g\in G}\alpha_g\ot 
    \lambda(g\ot - )) (x) \, .
$$
Therefore $[\alpha+\lambda,\alpha+\lambda]=2d^*(\beta)$ on 
$X_{0,3}$ if and only if 
Theorem~\ref{thm:main}(4) holds.

\subsection*{Condition (c): $[\alpha+\lambda, \beta]= 0$} 
On $X_{3,0}$ $X_{2,1}$, and $X_{1,2}$, the left side of
this equation is automatically 0 for degree reasons. We will compute values
on $X_{0,3}$. Similar to our previous calculation,
we find 
\[
   [\lambda,\beta ] = \lambda (\beta\ot 1) 
\quad \mbox{ and } \quad 
   [\alpha,\beta] = \beta  ((1\ot \sigma)(\alpha\ot 1) - 1\ot \alpha)
\]
on $(V\ot R)\cap (R\ot V)$. 
So $[\alpha+\lambda,\beta]= 0$ if and only
if $\beta ((1\ot\sigma)(\alpha\ot 1) - 1\ot \alpha) = 
- \lambda(\beta\ot 1) $. 
This is precisely Theorem~\ref{thm:main}(5).
\end{proof}


\section{Application:
Group actions on  polynomial rings}
\label{sec:CDL}

We now consider the special case when 
$S$ is the symmetric algebra $S(V)$ of a finite dimensional
$k$-vector space $V$.
Let $G$ be a finite group
acting on $S(V)$ by graded automorphisms.
Let $\chlk$ be the $k$-algebra
generated by the group ring $kG$ together with 
the vector space $V$ and subject to the relations
\begin{itemize}
\item
$gv-\, ^gv g -\lambda(g,v), \quad\ \  \text{ for }g\text{ in } G, \ v\text{ in } V$
\item
$vw-wv-\kappa(v,w), \quad\text{ for }v,w \text{ in } V,$
\end{itemize}
where
$$
\lambda:kG \times V \rightarrow kG, \quad
\kappa: V \times V \rightarrow kG \oplus (V\otimes kG)
$$
are bilinear  functions.
Letting $\kappa^C$ and $\kappa^L$ be the projections of $\kappa$
onto $kG$ and $V\ot kG$, respectively, 
$\chlk$ is the algebra $\chabl$ from earlier sections
with $\alpha=\kappa^L$ and $\beta=\kappa^C$.
Its homogeneous version 
is the algebra 
$$\text{HomogeneousVersion}(\chlk)=S(V)\rtimes G
=\mathcal{H}_{0,0}\ .$$ 
We say that $\chlk$ is a {\em Drinfeld orbifold algebra}
if it has the PBW property: 
$$\text{Gr}\,{\chlk} \cong S(V)\rtimes G
$$
as graded algebras.
Thus Drinfeld orbifold algebras 
are  PBW deformations of $S(V)\rtimes G$ .

In characteristic zero,
our definition of Drinfeld orbifold algebra coincides with that in~\cite{doa},
up to isomorphism, even though no parameter $\lambda$ appears there. 
This is a consequence of Theorem~\ref{thm:nonmod} in the next section:
In this nonmodular case, $\chlk$ is isomorphic to $\cH_{0,\kappa'}$ for some $\kappa'$.

The algebras $\mathcal{H}_{\lambda,\kappa}$ include as special cases many algebras 
of interest in the
literature, and our Theorem~\ref{RawPBWConditions} 
below unifies results giving necessary
and sufficient conditions on parameter functions for $\mathcal{H}_{\lambda,\kappa}$ to
have the PBW property. 
When $\lambda =0$ and $\kappa^L=0$, 
Drinfeld orbifold algebras  $\mathcal{H}_{0,\kappa}$
include Drinfeld's Hecke algebras~\cite{Drinfeld} and 
Etingof and Ginzburg's symplectic reflection algebras~\cite{EG}.
When $\lambda=~0$ and $\kappa^C=0$, Drinfeld orbifold algebras  $\mathcal{H}_{0,\kappa}$
exhibit a Lie type structure: Many of the conditions of 
Theorem~\ref{RawPBWConditions} below
are vacuous in this case, while
Condition~(3) states that $\kappa^L$ is $G$-invariant and
Conditions~(4) and~(6) are analogs of the Jacobi identity twisted by the group action. 
When $\kappa =0$, Drinfeld orbifold algebras $\mathcal{H}_{\lambda,0}$
include Lusztig's graded affine Hecke algebras \cite{Lusztig89}.

The following theorem simultaneously generalizes \cite[Theorem 3.1]{doa}
and \cite[Theorem 3.1]{ueber}. 

\begin{thm}
\label{RawPBWConditions}
Let $G$ be a finite group acting linearly on $V$, a finite dimensional
$k$-vector space.  Then $\chlk$
is a PBW deformation of $S(V)\rtimes G$ if and only if
\begin{itemize}
\item[(1)]
$\lambda(gh,v)=\lambda(g,\, ^hv)h+g\lambda(h,v)$,
\item[(2)]\rule{0ex}{4ex}
$
\kappa^C(\, ^gu,\, ^gv)g-g\kappa^C(u,  v)
\ =\ 
\lambda\big(\lambda(g,v), u\big)-\lambda\big(\lambda(g,u), v\big)
+\displaystyle{\sum_{a\in G}}\lambda\big(g,\kappa^L_a(u,v)\big)a\, ,%
$
\item[(3)]\rule{0ex}{4ex}
$\ ^g\big(\kappa^L_{g^{-1}h}(u,v)\big)
-\kappa^L_{hg^{-1}}(^gu,\ ^gv)
=
(^hv-\ ^gv)\lambda_h(g,u)-
(^hu-\ ^gu)\lambda_h(g,v)$, 
\item[(4)]\rule{0ex}{4ex}
\vspace{-3.3ex}
$$
\hspace{-17ex}
\begin{aligned}\hphantom{x}
0\ =\  &  \ 2\sum_{\sigma\in\Alt_3}
\kappa^C_g(v_{\sigma(1)},v_{\sigma(2)})(v_{\sigma(3)}-\, ^gv_{\sigma(3)})
\\
& +\ \sum_{\substack{a\in G\\ \sigma\in \Alt_3\rule{0ex}{1.5ex}}}
\kappa_{ga^{-1}}^L\big(v_{\sigma(1)}+\, ^a{v_{\sigma(1)}}, 
\kappa_a^L(v_{\sigma(2)},v_{\sigma(3)})\big)
\\
&-2\sum_{\substack{a\in G\\ \sigma\in \Alt_3\rule{0ex}{1.5ex}}} 
\kappa_a^L(v_{\sigma(1)},v_{\sigma(2)})\lambda_g(a,v_{\sigma(3)})\, ,
\end{aligned}
$$
\item[(5)]\rule{0ex}{3ex}
\vspace{-3.5ex}
$$
\hspace{-10ex}
\begin{aligned}
&2\sum_{\sigma\in\Alt_3}
\lambda\big(\kappa^C(v_{\sigma(1)},v_{\sigma(2)}),v_{\sigma(3)}\big)
\\ & \quad\quad\quad
= -\sum_{\substack{a\in G\\ \sigma\in \Alt_3\rule{0ex}{1.5ex}}} 
\kappa_{ga^{-1}}^C\big(v_{\sigma(1)}+\ ^av_{\sigma(1)}, 
\kappa_a^L(v_{\sigma(2)},v_{\sigma(3)})\big) ,
\end{aligned}
$$
\item[(6)]\rule[-2ex]{0ex}{3ex}
$0=
\kappa_g^L(u,v)(w-\ ^gw)+
\kappa_g^L(v,w)(u-\ ^gu)+
\kappa_g^L(w,u)(v-\ ^gv) $\, ,
\end{itemize}
in $S(V)\rtimes G$, for all $g,h$ in $G$ and all $u,v,w, v_1, v_2, v_3$ in $V$.
\end{thm}
\begin{proof}
The theorem follows from
Theorem~\ref{thm:main} by rewriting the conditions explicitly on
elements. 
\end{proof}

Alternatively, 
the conditions of the theorem follow from 
strategic and tedious application of
the Composition-Diamond Lemma
(such as in the proof of \cite[Theorem~3.1]{doa}).
Condition (1) follows from consideration of overlaps of the
form $ghv$ for $g,h$ in $G$, $v$ in $V$.
For Conditions~(2) and~(3),
we consider overlaps of the form $gwv$ for $w$ in $V$;
terms of degree~$1$ give rise to Condition~(3) while
those of degree~$0$ give rise to Condition~(2).
Overlaps of the form $uvw$ for $u$ in $V$
give the other conditions:
Terms of degree~$0$ give rise to Condition~(5),
terms of degree~$1$ give rise to Condition~(4), 
and terms of degree~$2$ give rise Condition~(6).
Note that we assume Condition~(6) to deduce Conditions~(4)
and~(5).

In the theorem above, 
we may set $\kappa^L = 0$ to obtain the conditions
of~\cite[Theorem~3.1]{ueber}
or instead set $\lambda = 0$
to obtain the conditions of~\cite[Theorem~3.1]{doa}.
Note that in Theorem~\ref{RawPBWConditions},
Condition~(3) measures the extent to which
$\kappa^L$ is $G$-invariant. Indeed,
 the failure
of $\kappa^L$ to be $G$-invariant is recorded by $d^*(\lambda)$, and 
so $\lambda$ is a cocycle if and only if $\kappa^L$
is invariant.
Condition~(3)  in particular implies that $\kappa_{1_G}^L$ is
$G$-invariant.

The conditions in the theorem also generalize a special case of 
Theorem~2.7 in \cite{Khare} by Khare:
He more generally considered actions of cocommutative algebras, while we
restrict to actions of group algebras $kG$. 
Khare more specifically restricted $\kappa^L$ to take values in the subspace
$V\cong V\ot k$ of $V\ot kG$.

We next give some examples of Drinfeld orbifold algebras.
The first example exhibits parameters $\kappa^C$, $\kappa^L$, and $\lambda$
all nonzero.  The second example shows that a new
class of deformations is possible in the modular setting;
see Remark~\ref{counterexample}.

\begin{example}
{\em Let $k$ have prime characteristic $p>2$, 
and 
$V=kv_1 \oplus kv_2\oplus kv_3$. 
Let $G\leq \text{GL}_3(k) $ be the cyclic group of order $p$
generated by the transvection $g$ in $\text{GL}(V)$ 
fixing $v_1,v_2$ and mapping $v_3$ to $v_1 + v_3$:
$$
G=\left<g=\left(\begin{smallmatrix} 1 & 0 & 1\\0 & 1 & 0 \\ 0 & 0 & 1  
\end{smallmatrix}\right)\right>\, .
$$ 
Define $$\lambda(g^i,v_3)=ig^{i-1},\
\kappa^C(v_1,v_3)=g=-\kappa^C(v_3,v_1),\ 
\kappa^L(v_1, v_3)=v_2=-\kappa^L(v_3, v_1),$$
and set $\lambda, \kappa^C, \kappa^L$ to be zero on all other pairs
of basis vectors. 
Then 
\[
\begin{aligned}
\chlk=T_{kG}(kG \ot V\ot kG)/&(gv_1-v_1g, \ gv_2-v_2g, \ gv_3-v_1g-v_3g-1,\\
& \hspace{.3cm}
v_1v_3-v_3v_1-v_2-g, \ v_1v_2-v_2v_1, \ v_2v_3-v_3v_2  )
\end{aligned}
\]
is a PBW deformation of $S(V)\rtimes G$
by Theorem~\ref{RawPBWConditions}.
}\end{example}

\begin{example}\label{counterexampleexplicit}
{\em
Let $k$ have prime characteristic $p>2$ and $V = kv\oplus kw$.
Suppose $G\leq {\rm{GL}}_2(k)$ is the cyclic group of order $p$ generated
by $g = \left(\begin{smallmatrix} 1 & 1 \\ 0 & 1 \end{smallmatrix}\right)$ so that
${}^g v =v$ and ${}^g w= v+w$.
Define 
\[
  \lambda( g^i,v) = i g^i,\
  \lambda(1,w)=  \lambda( g,w) = 0,\
  \lambda( g^i,w) = \tbinom{i}{2}\, g^i \text{ for }i>2,
\]
and $\kappa = 0$.
Then one may check the conditions of Theorem~\ref{RawPBWConditions} to conclude
that 
$$\mathcal{H}_{\lambda,0}
 = T_{kG}(kG\ot V\ot kG) / (gv-vg - g , \ gw-vg-wg , \ vw-wv)
$$ 
is a PBW deformation of $S(V)\rtimes G$.
}\end{example}

\section{Comparison of modular and nonmodular settings}

We now turn to the nonmodular setting, when the characteristic
of the underlying field $k$ does not divide the order of the acting
group $G$. We compare algebras modelled on Lusztig's
graded affine Hecke algebra~\cite{Lusztig89} 
to algebras modelled
on Drinfeld's Hecke algebra~\cite{Drinfeld} (such as the symplectic reflection 
algebras of  Etingof and Ginzburg~\cite{EG}).
The following theorem strengthens Theorem~4.1 of~\cite{ueber}
while simultaneously generalizing it to the setting
of Drinfeld
orbifold algebras (see~\cite{doa}) in the nonmodular setting.
The theorem was originally shown for Coxeter groups
and Lusztig's graded affine Hecke algebras in~\cite{RamShepler}. 

\begin{thm}\label{thm:nonmod}
Suppose $G$ acts linearly on a finite dimensional vector 
space $V$
over a field $k$ whose characteristic is coprime to $|G|$.
If the algebra $\mathcal{H}_{\lambda,\kappa}$ defined
in Section~\ref{sec:CDL} is a  PBW deformation
of $S(V)\rtimes G$ for some parameter functions
$$\lambda: kG\times V\rightarrow kG\quad
\text{and}\quad \kappa: V\times V \rightarrow kG \oplus 
(V\otimes kG),$$
then there exists a parameter function 
$$
\kappa': V\times V \rightarrow kG \oplus (V\otimes kG)$$
such that 
$$
\mathcal{H}_{\lambda, \kappa}\cong \mathcal{H}_{0,\kappa'}
$$
as filtered algebras and thus $\mathcal{H}_{0,\kappa'}$
also exhibits the PBW property.
\end{thm}
\begin{proof}
As in~\cite{ueber},
define $\gamma:V\otimes kG\rightarrow kG$ by
$$
\gamma(v\otimes g)=\tfrac{1}{|G|}\sum_{a,b\, \in G}
\lambda_{ab}\big(b, \, ^{b^{-1}}v\big) ag
=\sum_{a\in G}\ \gamma_a(v)\, ag
$$
for $\gamma_a:V \rightarrow k$ giving the
coefficient of $ag$ in $G$, and as
before, for each $h$ in $G$, $\lambda_h: kG\times V\rightarrow k$ is
defined by $\lambda(b, v) = \sum_{h\in G}\lambda_h(b, v) h$
for $b$ in $G$ and $v$ in $V$. 
We abbreviate $\gamma(u)$ for $\gamma(u\ot 1)$
for $u$ in $V$ in what follows for simplicity of notation. 
Define a parameter function $\kappa':V\times V
\rightarrow kG\oplus (V\otimes kG)$
by
$$
\begin{aligned}
\kappa'(u,v)\ =\ 
&\ \ \gamma(u)\gamma(v) -\gamma(v)\gamma(u)
+\lambda(\gamma(u), v)- \lambda(\gamma(v), u)\\
&+\kappa(u,v)- \kappa^L(u,v)\\
&+\tfrac{1}{|G|}\sum_{g\in G}
(1-\gamma)\big((\, ^g\kappa^L)(u, v)\big)
\end{aligned}
$$
for $u,v$ in $V$.  Here, $\kappa^L(u,v)$ is again the
degree $1$ part of $\kappa$, i.e.,
the projection of $\kappa(u,v)$
to $V\otimes kG$,
and we take the $G$-action on $\kappa^L$
induced from the action of $G$ on itself
by conjugation,
i.e.,
$(^g\kappa^L)(u,v)=\, ^g(\kappa^L(\, ^{g^{-1}}u,\, ^{g^{-1}}v))$
with $^g(v\otimes h)=\, ^gv \otimes ghg^{-1}$
for $g,h$ in $G$.

Let
$$
F=T_{kG}(kG\ot V\ot kG)
$$
and identify $v$ in $V$ with $1\ot v\ot 1$ in $F$. 
Define an algebra homomorphism
$$
f:F\rightarrow \mathcal{H}_{\lambda, \kappa}
\quad\text{ by }\quad
v\mapsto v + \gamma(v)
\text{ and }
g\mapsto g
\quad\text{for all } g\in G, v\in V,
$$
after identifying $\chlk$ with a quotient of $F$.
We will use Theorem~\ref{RawPBWConditions}
to verify that the relations defining $\mathcal{H}_{0,\kappa'}$
as a quotient of $F$ lie in the kernel of $f$.
It will follow that 
 $f$ extends to a filtered algebra homomorphism
$$f:\mathcal{H}_{0,\kappa'}\rightarrow \mathcal{H}_{\lambda,\kappa}\, .$$

We first check that 
elements $uv-vu-\kappa'(u,v)$ in $F$ for $u,v$ in $V$ 
are mapped to zero under $f$.
On one hand, 
$\kappa'(u,v)$ in $F$ is mapped under
$f$ to
$$
\kappa'(u,v)+\tfrac{1}{|G|}\sum_{g\in G}
\gamma\big((\, ^g\kappa^L)(u,v)\big).
$$
On the other hand, the commutator $[u,v]=uv-vu$ in $F$ maps to the
commutator 
$$[u+\gamma(u),v+\gamma(v)]
=
[u,v] + [\gamma(u), \gamma(v)]
+ \big(u\gamma(v)-\gamma(v)u-v\gamma(u)+\gamma(u)v\big)
$$ in $\chlk$.  But $[u,v]$ is $\kappa(u,v)$
in $\chlk$, and $\kappa'(u,v)$ by definition expresses
the commutator $[\gamma(u),\gamma(v)]$
in terms of $\kappa(u,v)$
and other terms.
Hence $[u,v]+[\gamma(u),\gamma(v)]$
simplifies
to
$$
\kappa'(u,v) - \lambda(\gamma(u),v)+\lambda(\gamma(v),u)
+\kappa^L(u,v)-\tfrac{1}{|G|}
\sum_{g\in G}\,(1-\gamma)\big((^g\kappa^L)(u, v)\big)
$$
in $\chlk$.  
We may also rewrite
$$
u\gamma(v)-\gamma(v)u-v\gamma(u)+\gamma(u)v
$$
as
$$
\lambda(\gamma(u),v)-\lambda(\gamma(v),u) 
-
\sum_{g\in G} \big(\gamma_g(v)(\,^gu-u)g
-\gamma_g(u)(\,^gv-v)g \big).
$$
Hence, the relation $uv-vu-\kappa'(u,v)$
in $F$ maps under $f$
to
\begin{equation}
\label{relationzero}
\kappa^L(u,v)
-\tfrac{1}{|G|}\sum_{g\in G}\,(^g\kappa^L)(u,v)
- \sum_{g\in G} \big(\gamma_g(v)(\,^gu-u)g
-\gamma_g(u)(\,^gv-v)g \big).
\end{equation}
We may then argue as in 
the proof of Theorem~4.1 of~\cite{ueber}
to show that Condition~(3) of Theorem~\ref{RawPBWConditions}
implies that
$$
\begin{aligned}
\sum_{g\in G} \big(\gamma_g(v)(\,^gu-u)g
-\gamma_g(u)(\,^gv-v)g \big)
& =\kappa^L(u,v)-\tfrac{1}{|G|}
\sum_{g,a\in G} \ ^g\big(\kappa^L_a(\,^{g^{-1}}u, \,^{g^{-1}}v)\big)gag^{-1}\\
&=\kappa^L(u,v)-\tfrac{1}{|G|}
\sum_{g\in G} \ (^g\kappa^L)(u, v)\, .
\end{aligned}
$$
Thus expression~(\ref{relationzero}) above is zero
and $uv-vu-\kappa'(u,v)$ lies in the kernel of $f$ for all
$u,v$ in $V$.

We may follow the rest of the proof of Theorem~4.1 of~\cite{ueber}
to see that $gv - \, ^g v g$ lies in the kernel of $f$
for all $g$ in $G$ and $v$ in $V$
and that $f$ is an isomorphism.
\end{proof}

\begin{remark}\label{counterexample}
{\em
Theorem~\ref{thm:nonmod} above is false in the modular setting,
i.e., when char~$(k)$ divides $|G|$.
Indeed, Example~\ref{counterexampleexplicit}
gives an algebra $\mathcal{H}_{\lambda,0}$
exhibiting the PBW property for some parameter function
$\lambda$, but we claim that there is no parameter
 $\kappa ' : V\times V \rightarrow kG\oplus (V\ot kG)$
for which $\mathcal{H}_{\lambda,0}\cong \mathcal{H}_{0, \kappa'}$ as filtered algebras.

If there were,
then $\mathcal{H}_{0,\kappa'}$ would exhibit the PBW property and
any isomorphism $f:  \mathcal{H}_{\lambda,0}\rightarrow \mathcal{H}_{0, \kappa'}$
would map the relation
$$gv-vg-g=gv-\, ^gvg-\lambda(g,v)=0$$ 
in $\mathcal{H}_{\lambda,0}$ to $0$ in 
$\mathcal{H}_{0, \kappa'}$.
But $f$ is an algebra homomorphism
and takes the filtered degree $1$ component of 
$\mathcal{H}_{\lambda,0}$ to that of $\mathcal{H}_{0, \kappa'}$,
giving a relation
$$f(g) f(v) -f(v)f(g)-f(g)=0$$ 
in $ \mathcal{H}_{0, \kappa'}$ with first two
terms of the left hand side of filtered degree $1$. 
In particular, the sum of the terms
of degree $0$ vanish.  But this implies that $f(g)=0$
since the degree $0$
terms of $f(g)f(v)-f(v)f(g)$ cancel with each other
as $kG$ is commutative.
This contradicts the assumption that $f$ is an isomorphism. }
\end{remark}


\end{document}